\documentclass[11pt]{amsart}
\usepackage{amssymb,amscd,amsthm,verbatim}
\usepackage{amsbsy}
\usepackage{latexsym}
\usepackage[all,cmtip]{xy}
\usepackage[mathscr]{euscript}
\usepackage{bbm}

\date{\today}

\newcommand{\Z}{{\mathbb Z}}
\newcommand{\R}{{\mathbb R}}

\newcommand{\bbR}{{\mathbbm{R}}}

\newcommand{\bbZ}{{\mathbbm{Z}}}

\newtheorem{theorem}{Theorem}[section]
\newtheorem{lemma}[theorem]{Lemma}

\newtheorem{coro}[theorem]{Corollary}

\theoremstyle{definition}

\newtheorem{remark}[theorem]{Remark}
\theoremstyle{plain}

\allowdisplaybreaks \numberwithin{equation}{section}

\newcommand{\set}[1]{\left\{#1\right\}}

\begin{document}

\title[Limit-Periodic Operators With Lipschitz Continuous IDS]{Limit-Periodic Schr\"odinger Operators With Lipschitz Continuous IDS}

\author[D.\ Damanik]{David Damanik}
\address{Department of Mathematics, Rice University, Houston, TX~77005, USA}
\email{damanik@rice.edu}
\thanks{D.D.\ was supported in part by NSF grants DMS--1361625 and DMS--1700131.}

\author[J.\ Fillman]{Jake Fillman}
\address{Department of Mathematics, Virginia Polytechnic Institute and State University, 225 Stanger Street, Blacksburg, VA~24061, USA}
\email{fillman@vt.edu}
\thanks{J.F.\ was supported in part by an AMS-Simons travel grant, 2016--2018}

\begin{abstract}
We show that there exist limit-periodic Schr\"odinger operators such that the associated integrated density of states is Lipschitz continuous. These operators arise in the inverse spectral theoretic KAM approach of P\"oschel.
\end{abstract}

\maketitle

\section{Introduction}

The main result of this paper is the following.

\begin{theorem}\label{t.main}
There are limit-periodic Schr\"odinger operators with a Lip\-schitz continuous integrated density of states.
\end{theorem}

The main point of this result is that it exhibits a new phenomenon. To the best of our knowledge there were no previous examples of Schr\"odinger operators with almost periodic potentials such that the associated integrated density of states (IDS) has such a strong regularity property. On the H\"older scale, the best previously known result states $\frac12$-H\"older continuity of the integrated density of states under suitable assumptions; compare, for example, \cite{AJ10, HA09}. We refer the reader also to \cite{B00, GS01, YZ14} for other results on the H\"older regularity of the IDS for almost periodic Schr\"odinger operators.

The  $\frac12$-H\"older continuity results mentioned above are typically optimal in the context in which they are established, as demonstrated by the presence of square-root singularities at the endpoints of the gaps of the spectrum (see, e.g., \cite{P06}). This explains why a result like Theorem~\ref{t.main} was not only not known, but is indeed quite surprising.

\medskip

Let us now make the terminology more precise. For the sake of simplicity we will work in the context of \emph{discrete one-dimensional Schr\"odinger operators}, that is, operators of the form
\begin{equation}\label{e.oper}
[H \psi](n) = \psi(n+1) + \psi(n-1) + V(n) \psi(n)
\end{equation}
acting in $\ell^2(\Z)$ with a \emph{potential} $V : \Z \to \R$. We write $H_V$ for $H$ if we want to emphasize the dependence on the potential. The potential $V$ is called \emph{almost periodic} if it is bounded and the set of its translates is relatively compact in $\ell^\infty(\Z)$. In other words,
$$
\Omega = \overline{ \{ V( \cdot - m) : m \in \Z \} }^{\| \cdot \|_\infty}
$$
is a compact subset of $\ell^\infty(\Z)$. It turns out that $\Omega$ is a compact abelian group, where the group structure is the one obtained by continuous extension of the natural group structure (induced by $\Z$) on the set of translates. Thus, $\Omega$ carries a natural measure, namely normalized Haar measure, which we denote by $\mu$. Then, the shift $S\omega = \omega(\cdot - 1)$ defines a minimal homeomorphism from $\Omega$ to itself, and the dynamical system $(\Omega,S)$ is uniquely ergodic (with $\mu$ being the unique preserved probability measure).

Note that each element $\omega$ of $\Omega$ belongs to $\ell^\infty(\Z)$ and hence can serve as the potential of a Schr\"odinger operator:
\begin{equation}\label{e.operfam}
[H_\omega \psi](n) = \psi(n+1) + \psi(n-1) + \omega(n) \psi(n)
\end{equation}
for $\psi \in \ell^2(\Z)$.

The \emph{integrated density of states} $k : \R \to \R$ (henceforth \emph{IDS}) can be introduced in many equivalent ways. Let us use the following definition, which is the most convenient for what follows later. By a standard result from the theory of ergodic Schr\"odinger operators \cite{CL90, CFKS87, DF18}, there exists a measure, denoted by $dk$ and called the \emph{density of states measure}, such that for every continuous function $g:\bbR \to \bbR$ and $\mu$-almost every $\omega \in \Omega$, we have
\begin{equation}\label{e.dosmeas}
\int g \, dk = \lim_{N \to \infty} \frac1N \sum_{n = 0}^{N-1} \langle \delta_n, g(H_\omega) \delta_n \rangle.
\end{equation}
Moreover, in the current setting, we know that \eqref{e.dosmeas} holds simultaneously for every continuous $g$ and every $\omega \in \Omega$ by unique ergodicity.

The distribution function $k$ of the density of states measure is the IDS associated with the family \eqref{e.operfam} (resp., the initial operator \eqref{e.oper}). It is also a standard result from the general theory that $k$ is always continuous \cite{CL90, CFKS87, DF18}, and many papers have been devoted to the investigation of properties of $k$ that go beyond this basic result.

\medskip

A potential $V$ is called \emph{periodic} if there is a \emph{period} $p \in \Z_+$ such that $V(\cdot - p) = V(\cdot)$; it is called \emph{limit-periodic} if it lies in the $\ell^\infty$-closure of the space of periodic potentials. It is well known, and not hard to see, that every limit-periodic $V$ is almost periodic in the sense above.

\medskip

Having recalled the necessary notions, it is now clear what our objective is. We wish to prove the existence of a limit-periodic $V$ such that for a suitable constant $C$, we have
\begin{equation}\label{e.goal}
\int \chi_I \, dk \le C |I|
\end{equation}
for any interval $I \subseteq \R$, where $|I|$ denotes the length of the interval $I$.

We will make crucial use of an inverse spectral theoretic KAM approach to limit-periodic Schr\"odinger operators due to P\"oschel \cite{P83}.
The relevant details will be described in Section~\ref{s.2}. The results from this paper allow us to estimate the left-hand side of \eqref{e.goal} in terms of quantities that are variants of a geometric series and hence can be computed explicitly. The necessary calculations are provided in Section~\ref{s.3}. With this explicit estimate in hand, the proof of Theorem~\ref{t.main} can then be concluded in Section~\ref{s.4}.

\medskip

Let us end the introduction with a few general remarks:

(1) We consider the one-dimensional case for convenience. P\"oschel's work \cite{P83} can treat limit-periodic Schr\"odinger operators in $\ell^2(\Z^d)$ for arbitrary $d \in \Z_+$. However, our main goal in this paper is to exhibit a new phenomenon and we have decided to do so in the simplest setting.

(2) The obstruction to an improvement of a $\frac12$-H\"older continuity result for the IDS holds fairly generally in the regime of absolutely continuous spectral measures. Concretely, if an almost periodic Schr\"odinger operator $H_V$ has purely absolutely continuous spectral type and its spectrum $\Sigma$ is a homogeneous set in the sense of Carleson \cite{carleson83}, then its density of states measure coincides with the equilibrium measure of $\Sigma$. Then, the IDS can be no better than $\frac12$-H\"older continuous by standard product formulae for the equilibrium measure; compare \cite[Eq.~(6.4.2)]{SY97}.

(3) The known classes of almost periodic Schr\"odinger operators with singular continuous spectral measures typically come with a rather thin spectrum (e.g., of zero Lebesgue measure) or are dual to operators with absolutely continuous spectrum. In either scenario, one would not expect the IDS to be very regular. It would be nice to prove a general result to this effect.

(4) The above remarks suggest that a necessary condition for a very regular (e.g., Lipschitz continuous) IDS in the context of almost periodic Schr\"odinger operators may be the pure point nature of the spectral measures.\footnote{Clearly, pure point spectrum is not sufficient, which may be seen by considering, e.g., the supercritical almost Mathieu operator with Diophantine frequency.} Indeed, the examples we construct have pure point spectrum. Actually, they are even uniformly localized, which is a much stronger property that does not occur very often. Concretely, uniform localization is known to fail for the Anderson model and the almost Mathieu operator (in the localized regime)~\cite[Appendix~C]{dRJLS1996}.

(5) Our result may be contrasted with the celebrated estimate of Wegner \cite{Wegner1981} for the Anderson model: Suppose $\{\omega(n)\}_{n \in \bbZ}$ is sequence of i.i.d.\ random variables in $\bbR$ whose common distribution (often called the \emph{single-site distribution}) is absolutely continuous with bounded density $g$. The associated IDS is Lipschitz continuous; indeed,
\[
|k(E_1) - k(E_2)|
\leq
C \| g\|_\infty |E_1 - E_2|
\]
for a constant $C$. For expository presentations, see \cite[Chapter~4]{AW2015} and \cite[Section~5]{Kirsch2008}. On the other hand, Lipschitz continuity may fail if the single-site distribution is too singular; this is proved explicitly in the strongly coupled Bernoulli case in \cite{CKM87}, but as pointed out there, the argument extends to some more general cases. In any event, even in the Anderson model, Lipschitz continuity of the IDS does not always hold and needs sufficient regularity of the single-site distribution.

\section{P\"oschel's Inverse Spectral Results}\label{s.2}

In this section we recall P\"oschel's approach to the spectral analysis of limit-periodic Schr\"odinger operators in the large-coupling regime \cite{P83}.
We focus on the inverse spectral aspect of this approach, as it is this aspect that we will employ below.

First, we define the notion of an \emph{approximation function}, which is P\"oschel's mechanism for making precise the notion of having divisors that decay ``not too quickly.'' Given a function $\Omega:[0,\infty) \to [0,\infty)$,
define
\begin{align*}
\Phi(t)
& :=
t^{-4} \sup_{r \ge 0} \Omega(r) e^{-t r} \\
\kappa_t
& :=
\set{\{t_j\}_{j = 0}^\infty : t \geq t_0 \geq t_1 \geq \cdots  \text{ and } \sum_{j = 0}^\infty t_j \leq t} \\
\Psi(t)
& :=
\inf_{\kappa_t} \prod_{j = 0}^\infty \Phi(t_j)2^{-j-1}
\end{align*}
for $t > 0$. We say that $\Omega$ is an \emph{approximation function} if $\Phi(t)$ and $\Psi(t)$ are finite for every $t > 0$. One can check that $\Omega(r) = r^\alpha$
is an approximation function for each $\alpha \ge 0$.

Now, suppose that $\mathcal{M} \subseteq \ell^\infty(\bbZ)$ is a Banach subalgebra with respect to pointwise addition and pointwise multiplication (in particular, $\mathcal{M}$ is assumed to contain the constant sequence 1).

Recall the shift on $\ell^\infty(\bbZ)$ is denoted by $S\omega = \omega(\cdot -1)$. We say that $\lambda : \bbZ \to \bbR$ is a \emph{distal sequence for} $\mathcal{M}$ if $\big( \lambda - S^k\lambda \big)^{-1} \in \mathcal{M}$ for each $k \neq 0$ and
\[
\left\| \big(\lambda - S^k \lambda\big)^{-1} \right\|_\infty
\leq
\Omega(|k|),
\quad
\text{for all } k \neq 0,
\]
where $\Omega$ is an approximation function.\footnote{Here, the inverse refers to the multiplicative inverse in the Banach algebra, where multiplication is taken to mean pointwise multiplication.}

\begin{theorem}[P\"oschel, 1983 \cite{P83}]
If $\lambda$ is a distal sequence for $\mathcal{M}$, then there exists $\varepsilon_0 > 0$ such that for any $0< \varepsilon \leq \varepsilon_0$, there is a sequence $V$ so that $\lambda - V \in \mathcal{M}$ and the Schr\"odinger operator $H_{\varepsilon^{-1}V} = \Delta + \varepsilon^{-1} V$ is spectrally localized with eigenvalues $\{ \varepsilon^{-1} \lambda_j : j \in \bbZ\}$. Moreover, if $\psi_k$ is the normalized eigenvector corresponding to the eigenvalue $\lambda_k$, then there are constants $c>0$ and $d>1$ such that
\[
|\psi_k(n)|^2
\leq
c d^{-|k-n|}
\]
for all $k$ and $n$.
\end{theorem}

\noindent \textbf{P\"oschel's Example: A Limit-Periodic Distal Sequence.} Let $\mathcal{P}_n$ denote the set of sequences in $\ell^\infty(\bbZ)$ having period $2^n$, and $\mathcal{P} = \bigcup_n \mathcal{P}_n$; the space
\[
\mathcal{L}
=
\overline{\mathcal{P}}
\]
is a Banach algebra and a subspace of the space of all limit-periodic sequences. We may construct a distal sequence for $\mathcal{L}$ as follows. For $j \in \bbZ_+$, define the set $A_j$ by
\[
A_j
:=
\begin{cases}
\bigcup_{N \in \bbZ} [N \cdot 2^j, N\cdot 2^j + 2^{j-1}), & j \text{ even};\\[2mm]
\bigcup_{N \in \bbZ} [N \cdot 2^j+2^{j-1}, (N+1)\cdot 2^j), & j \text{ odd}.
\end{cases}
\]
For example, $A_1 = 2\Z+1$, the collection of odd integers. Let $a_j = \chi_{A_j}$ denote the characteristic function of $a_j$. Then, the sequence
\[
\lambda_n
=
\sum_{j=1}^\infty a_j(n) 2^{-j}
\]
belongs to $\mathcal{L}$; moreover, the inequality
\[
\left\| \big(\lambda - S^k\lambda \big)^{-1} \right\|
\leq
16|k|
\]
for $k \neq 0$ means that $\lambda$ is distal.

In order to prove Theorem~\ref{t.main}, we will need to characterize the integers for which $\lambda_n$ belongs to dyadic intervals of the form
\begin{equation} \label{eq:dyadicIntDef}
I_{m,j}
:=
\left[\frac{j}{2^m}, \frac{j+1}{2^m} \right).
\end{equation}

\begin{lemma} \label{l:dyadicLanding}
For any $m \in \bbZ_+$ and any integer $0 \leq j < 2^m$, there is an integer $\ell = \ell(j,m)$ so that
\[
\lambda_k \in I_{m,j}
\iff
k \in \ell + 2^m \bbZ.
\]
\end{lemma}

\begin{proof}
Given $k \in \bbZ$, let us first note that the sequence $\{a_j(k)\}_{j=1}^\infty$ cannot terminate in an infinite string of ones. Concretely, suppose $k \geq 0$ and choose $r$ even and large enough that $2^r > k$. Then, for every odd $s > r$, one has $k \in [0,2^{s-1})$, whence $k \notin A_s$. The argument for $k < 0$ is similar.

Given $j$ and $m$, let $j = \sum_{i=1}^{m} b_i 2^{i-1}$ with $b_i \in\{0, 1\}$. Then, since the sequence $\{a_i(k)\}$ cannot terminate in an infinite string of $1$'s, it follows that $\lambda_k \in I_{m,j}$ if and only if $a_i(k) = b_i$ for every $1 \le i \le m$. By \cite[Lemma~2.1]{P83}, there is a unique $\ell \in [0,2^m)$ such that
\[
(a_1(\ell), \ldots a_m(\ell))
=
(b_1,\ldots,b_m).
\]
By the definition of the $a_j$, $(a_1,\ldots,a_m)$ is a $2^m$-periodic map from $\bbZ$ to $\{0,1\}^m$, which concludes the proof of the lemma.
\end{proof}

\section{Fun and Games with Geometric Series}\label{s.3}

In this section we prove some basic statements about quantities related to geometric series. These results will allow us in the next section to estimate the weight assigned by the density of states measure of the operators in question to suitable intervals.

\begin{lemma} \label{l:latticesum}
Let $d > 1$, $\Delta > 0$. Then, for all $x \in \bbR$,
\begin{equation} \label{eq:latticesum}
\sum_{j \in \bbZ} d^{-|x-j\Delta|}
=
\frac{d^{-s} + d^{-(\Delta - s)}}{1-d^{-\Delta}},
\end{equation}
where $s := \mathrm{dist}(x,\Delta\bbZ)$.
\end{lemma}

\begin{proof}
Let $g(x)$ denote the left-hand side of \eqref{eq:latticesum}, and note that
\[
g(x+\Delta) \equiv g(x),
\]
so it suffices to consider $x \in [0,\Delta)$; for such $x$, one has
\[
\begin{cases}
x - j\Delta \geq 0 & \text{whenever } j \leq 0 \\
x - j\Delta \leq 0 & \text{whenever } j \geq 1.
\end{cases}
\]
Then,
\begin{align*}
g(x)
& =
\sum_{j=1}^\infty d^{-(j\Delta - x)} + \sum_{j=-\infty}^0 d^{-(x-j\Delta)} \\[1.5mm]
& =
\frac{d^{-(\Delta - x)}}{1-d^{-\Delta}} + \frac{d^{-x}}{1-d^{-\Delta}} \\[1.5mm]
& =
\frac{d^{-s} + d^{-(\Delta-s)}}{1 - d^{-\Delta}}.
\end{align*}
In the final line we used that $s=x$ for $x \in [0,\Delta/2)$ and $s = \Delta - x$ for $x \in [\Delta/2,\Delta)$.
Thus, the right-hand side of \eqref{eq:latticesum} and $g$ coincide for $x \in [0,\Delta)$. Since the right-hand side of \eqref{eq:latticesum} is clearly $\Delta$-periodic in $x$, we are done.
\end{proof}

\begin{coro}
For $d>1$, $m \in \bbZ_+$, $\ell \in \bbZ$, we have
\begin{equation} \label{eq:eigEst}
\lim_{N \to \infty} \frac{1}{N} \sum_{j\in \bbZ} \sum_{n=0}^{N-1} d^{-|n - \ell - j2^m|}
=
2^{-m} \cdot \frac{1+d^{-1}}{1-d^{-1}}.
\end{equation}
\end{coro}

\begin{proof}
For $n \in \bbZ$, Lemma~\ref{l:latticesum} yields
\[
\phi(n)
:=
\sum_{j \in \bbZ} d^{-|n - \ell - j2^m|}
=
\frac{d^{-s} + d^{-(2^m-s)}}{1 - d^{-2^m}},
\]
where $s = \mathrm{dist}(n - \ell,2^m\bbZ)$. As observed in Lemma~\ref{l:latticesum}, one has $\phi(n+2^m) \equiv \phi(n)$ and hence the limit on the left-hand side of \eqref{eq:eigEst} exists and satisfies
\begin{equation} \label{eq:avglim}
\lim_{N \to \infty} \frac{1}{N} \sum_{j\in \bbZ} \sum_{n=0}^{N-1} d^{-|n - \ell - j2^m|}
=
2^{-m} \sum_{n=0}^{2^m-1} \phi(n).
\end{equation}
As $n$ ranges from $0$ to $2^m-1$, $s$ attains all values in $(0,2^{m-1}) \cap \bbZ$ twice and attains the values $0$ and $2^{m-1}$ once each; note this means that $2^m-s$ attains each value in $[2^{m-1},2^m]\cap \bbZ$ twice except for the endpoints. Thus, we get
\begin{align*}
\sum_{n=0}^{2^m-1} \phi(n)
& =
\frac{1}{1-d^{-2^m}} \left( \left( 2\sum_{i=0}^{2^m} d^{-i} \right) -1 - d^{-2^m} \right) \\[2mm]
& =
\frac{1}{1-d^{-2^m}} \left( \frac{2(1-d^{-2^m-1})}{1-d^{-1}} - (1+d^{-2^m}) \right) \\[2mm]
& =
\frac{1+d^{-1}}{1-d^{-1}}.
\end{align*}
In view of \eqref{eq:avglim}, we are done.
\end{proof}

\section{Proof of Theorem~\ref{t.main}}\label{s.4}

In this section we prove Theorem~\ref{t.main}. In the inverse spectral theoretic P\"oschel approach, we will choose the initial distal sequence so that for $\varepsilon$ small enough, the resulting limit-periodic operator satisfies the desired estimate \eqref{e.goal} with a suitable constant $C$ after rescaling the energy to account for the $\varepsilon$ factor. We will see that the primary object that determines the size of $C$ is the uniform exponential decay rate of the eigenfunctions.

\medskip

Since it is important for our purposes, let us briefly note that (in the present setting) \eqref{e.dosmeas} holds for all $\omega \in \Omega$ with $g$ replaced by $\chi_I$; that is,
\begin{equation} \label{eq:dosmeaschi}
\int \chi_I \, dk
=
\lim_{N \to \infty} \frac1N \sum_{n = 0}^{N-1} \langle \delta_n, \chi_I(H_\omega) \delta_n \rangle.
\end{equation}
To see this, let $\varepsilon > 0$ and choose continuous functions $f \le \chi_I \le g$ so that
\[
\int (\chi_I - f) \, dk < \varepsilon
\quad
\text{and}
\int (g - \chi_I) \, dk < \varepsilon;
\]
note that this already uses continuity of $dk$. Then, for $N$ large enough,
\begin{align*}
\frac{1}{N} \sum_{n=0}^{N-1} \langle \delta_n, \chi_I(H_\omega) \delta_n \rangle
& \leq
\frac{1}{N} \sum_{n=0}^{N-1} \langle \delta_n, g(H_\omega) \delta_n \rangle \\
& <
\int g \, dk + \varepsilon \\
& <
\int \chi_I \, dk + 2 \varepsilon.
\end{align*}
Similarly, approximating $\chi_I$ with $f$ instead of $g$, one gets
\[
\frac{1}{N} \sum_{n=0}^{N-1} \langle \delta_n, \chi_I(H_\omega) \delta_n \rangle
>
\int \chi_I \, dk - 2\varepsilon
\]
for all sufficiently large $N$; thus, \eqref{eq:dosmeaschi} follows.

\begin{remark}
Alternatively, one may start with the statement that, for every $I$, \eqref{eq:dosmeaschi} holds for $\mu$-almost every $\omega$, which follows readily from Birkhoff's ergodic theorem. Picking an $\omega$ for which \eqref{eq:dosmeaschi} holds, we may use the extension of P\"oschel's work to \emph{all} $\omega \in \Omega$ by Damanik and Gan \cite{DG11, DG13} and proceed in a similar fashion as below, but with $H$ replaced by $H_\omega$.
\end{remark}

\begin{proof}[Proof of Theorem~\ref{t.main}]
Let $\lambda$ denote the distal sequence defined in Section~\ref{s.2} and $\varepsilon > 0$ sufficiently small. Then, as discussed in Section~\ref{s.2}, there is a limit-periodic potential $V$ so that $H := \Delta + \varepsilon^{-1} V$ has eignevalues $\{\varepsilon^{-1} \lambda_k\}_{k \in \bbZ}$ and a complete set of eigenfunctions $\{\psi_k\}$ so that $\varepsilon H \psi_k = \lambda_k \psi_k$ and
\begin{equation} \label{eq:eigDecay}
|\psi_k(n)|^2
\leq
c d^{-|n-k|}
\end{equation}
for constants $c>0$ and $d>1$.

\medskip



As already mentioned, our goal is to prove \eqref{e.goal} with a suitable constant $C$. It clearly suffices to do this for every (rescaled) dyadic interval $I$ of the form
\begin{equation}\label{e.intervallength}
I
=
\varepsilon^{-1} I_{m,j},
\end{equation}
where $I_{m,j} = [j2^{-m},(j+1)2^{-m})$ as in \eqref{eq:dyadicIntDef}. For convenience, we define $E_k = \varepsilon^{-1} \lambda_k$ so that $E_k \in I$ if and only if $\lambda_k \in I_{m,j}$. Thus, by Lemma~\ref{l:dyadicLanding}, there exists $\ell = \ell(I)$ so that
\begin{equation}\label{e.eigenvalues}
E_k \in I \Leftrightarrow k \in \{\ell + j 2^m : j \in \Z \}.
\end{equation}
Recall that by \eqref{eq:dosmeaschi}, we have
\begin{equation}\label{e.dosmeas2}
\int \chi_I \, dk = \lim_{N \to \infty} \frac1N \sum_{n = 0}^{N-1} \langle \delta_n, \chi_I(H) \delta_n \rangle.
\end{equation}

Next note that
\begin{align}
\label{e.expansion}
\langle \delta_n, \chi_I(H) \delta_n \rangle
& =
\Big\langle \delta_n, \sum_{E_{k} \in I}\langle \psi_{k}, \delta_n \rangle \psi_{k} \Big\rangle \\
\nonumber & =
\sum_{E_{k} \in I} \langle \psi_{k}, \delta_n \rangle \langle \delta_n,  \psi_{k} \rangle \\
\nonumber & = \sum_{E_{k} \in I} |\psi_{k}(n)|^2.
\end{align}

Thus, we have
\begin{align*}
\int \chi_I \, dk
& =
\lim_{N \to \infty} \frac1N \sum_{n = 0}^{N-1} \sum_{E_{k} \in I} |\psi_{k}(n)|^2 \\
& \le \limsup_{N \to \infty} \frac1N \sum_{n = 0}^{N-1} \sum_{j \in \bbZ} c \cdot d^{-|n-\ell - j 2^m|} \\
& = c \cdot 2^{-m} \cdot \frac{1+d^{-1}}{1-d^{-1}} \\
& = c \cdot \varepsilon \cdot \frac{1+d^{-1}}{1-d^{-1}} \cdot |I|.
\end{align*}
Here we used \eqref{e.dosmeas2} and \eqref{e.expansion} in the first step, \eqref{eq:eigDecay} and \eqref{e.eigenvalues} in the second step, \eqref{eq:eigEst} in the third step, and \eqref{e.intervallength} in the fourth step.

This proves \eqref{e.goal} with the constant $c \cdot \varepsilon \cdot \frac{1+d^{-1}}{1-d^{-1}}$ for every rescaled dyadic interval $I$ of length $\varepsilon^{-1} 2^{-m}$, which in turn implies \eqref{e.goal} with the same constant for every interval $I$, concluding the proof.
\end{proof}

\end{document}